\documentclass[pdftex,12pt]{amsart}
\usepackage[dvips]{graphicx,color}
\usepackage[pdftex,colorlinks=true, pdfstartview=FitV, linkcolor=blue, citecolor=blue, urlcolor=blue]{hyperref}%\usepackage{a4wide}
\usepackage{amsthm}
\usepackage{amssymb}
\usepackage{amsmath}
\usepackage{amsfonts,amsmath,amssymb,amsthm,mathrsfs,verbatim,hyperref}

\usepackage[all,cmtip]{xy}
\usepackage{latexsym}
\usepackage{graphicx}
\newtheorem{theorem}{Theorem}[section]
\newtheorem{lemma}[theorem]{Lemma}
\newtheorem{corollary}[theorem]{Corollary}
\newtheorem{proposition}[theorem]{Proposition}

\theoremstyle{definition}

\newtheorem{definition}[theorem]{Definition}
\newtheorem{example}[theorem]{Example}

\newtheorem{remark}[theorem]{Remark}

\renewcommand{\theenumi}{{\alph{enumi}}}

\newcommand{\ble}{\begin{lemma}}
\newcommand{\ele}{\end{lemma}}
\newcommand{\bth}{\begin{theorem}}
\renewcommand{\eth}{\end{theorem}}
\newcommand{\bpr}{\begin{proposition}}
\newcommand{\epr}{\end{proposition}}
\newcommand{\bco}{\begin{corollary}}
\newcommand{\eco}{\end{corollary}}
\newcommand{\bde}{\begin{definition}}
\newcommand{\ede}{\end{definition}}
\newcommand{\bre}{\begin{remark}}
\newcommand{\ere}{\end{remark}}
\newcommand{\bex}{\begin{example}}
\newcommand{\eex}{\end{example}}
\newcommand{\beq}{\begin{equation}}
\newcommand{\eeq}{\end{equation}}
\newcommand{\pf}{\noindent{\bf Proof}\hspace{7pt}}

\newcommand{\into}{\hookrightarrow}
\newcommand{\sbe}{\subseteq}
\newcommand{\De}{\Delta}
\newcommand{\bE}{{\bf E}}
\newcommand{\bI}{{\bf I}}
\newcommand{\cA}{{\mathcal A}}
\newcommand{\cC}{{\mathcal C}}
\newcommand{\cG}{{\mathcal G}}
\newcommand{\cL}{{\mathcal L}}
\DeclareMathOperator{\Aut}{Aut}
\renewcommand{\bar}{\overline}
\DeclareMathOperator{\id}{id}
\DeclareMathOperator{\im}{im}
\DeclareMathOperator{\Sym}{Sym}
\newcommand{\FF}{{\mathbb F}}
\newcommand{\NN}{{\mathbb N}}
\newcommand{\ZZ}{{\mathbb Z}}
\newcommand{\after}{\mathbin{ \circ }}

\DeclareMathOperator{\Stab}{Stab}
\newcommand{\sM}{{\mathsf{M}}}
\newcommand{\sT}{{\mathsf{T}}}
\newcommand{\mn}{\ \medskip \newline }
\renewcommand{\qed}{\hfill $\square$}

\newcommand{\ead}{\email}
\newenvironment{keyword}{Keywords: \keywords}{}
\newcommand{\MSC}[1]{MSC #1:\subjclass}
\newcommand{\sep}{\hspace{1ex}}

\begin{document}
\title{Coxeter-Chein Loops}
\author{Rieuwert J. Blok}
\address[Rieuwert J. Blok]{department of mathematics and statistics\\
bowling green state university\\
bowling green, oh 43403\\
u.s.a.}
\ead{blokr@member.ams.org}

\author{ Stephen Gagola III}
\address[ Stephen Gagola III]{department of mathematics and statistics\\
bowling green state university\\
bowling green, oh 43403\\
u.s.a.}
\ead{	sgagola@bgsu.edu}
   
\maketitle

\begin{abstract}
In 1974 Orin Chein discovered a new family of Moufang loops 
which 
are now called Chein loops. Such a loop can be created from any group $W$ together with $\ZZ_2$ by a variation on a semi-direct product.
We study these loops in the case where $W$ is a Coxeter group and show that it has what we call a Chein-Coxeter system, a small set of generators of order $2$, together with a set of relations closely related to the Coxeter relations and Chein relations.
As a result we are able to give amalgam presentations for Coxeter-Chein loops. This is to our knowledge the first such presentation for a Moufang loop.
\end{abstract}
\begin{keyword}
Coxeter groups, amalgams, Moufang loops, generators and relations.

\MSC{2010} 20F05 \sep 20F55 \sep 20N05\sep 51E99.
\end{keyword}

%\newpage
%\pagestyle{plain}
\section{introduction}
All Moufang loops considered in this article are finite. However, the definitions
below may still hold for loops of infinite order.
A {\em quasigroup} is a nonempty set $S$ with a closed binary operation,
$$(x,y)\mapsto x\cdot y,$$ such that:
\begin{enumerate}
\item  $a\cdot x = b$ determines a unique element $x\in S$ given $a, b\in S$; and
\item $b = y\cdot a$ determines a unique element $y\in S$ given $a,b\in S$.
\end{enumerate}
A {\em loop} is a quasigroup $L$ with a 2-sided identity element. It was Ruth Moufang who introduced what we now call the Moufang identity,
\mn
$\begin{array}{@{}l@{\hspace{.3\textwidth}} rll}
{\rm (m1)} & z(x(yx)) & = ((zx)y)x & \mbox{ for all }x,y,z\in L\\
\end{array}
$
\mn
One can show (see Lemma 3.1 of~\cite{Bru1958}) that in any loop (m1) is actually
equivalent to:
\mn
$\begin{array}{@{}l@{\hspace{.3\textwidth}} rl}
{\rm (m2)} & x(y(xz)) &= ((xy)x)z\\
{\rm (m3)} & (xy)(zx)  &= x(yz)x,\\
\end{array}
$
\mn
for all $x,y,z\in L$.

\bde\label{dfn:Moufang loop}
 A {\em Moufang loop} is a loop, $L$, that satisfies the Moufang identities.
\ede

\bde\label{dfn:chein loop}
Given a group $G$ the {\em Chein loop}~\cite{Che1974} over $G$, denoted $M(G,2)$ is the set $G\uplus Gu$, where $u$ is some formal element not in $G$ together with a binary operation that extends that of $G$ and satisfies
\mn
$\begin{array}{@{}l@{\hspace{.3\textwidth}} rl}
{\rm (c1)} & g_1(g_2 u)&=(g_2g_1)u,\\
{\rm (c2)} & (g_1 u)g_2&=(g_1g_2^{-1})u,\\
{\rm (c3)} & (g_1u)(g_2u)&=g_2^{-1}g_1,\\
\end{array}
$
\mn
for all $g_1,g_2\in G$.
\ede
\bde
A {\em Coxeter diagram over the index set $I=\{1,2,\ldots,n\}$} is a symmetric matrix $\sM=(m_{ij})_{i,j\in I}$ with entries in $\NN_{\ge 1}\cup \{\infty\}$ such that 
 $m_{ii}=1$ for all $i\in I$. 
 
 A {\em Coxeter system with diagram $\sM$} is a pair $(W,\{s_i\}_{i\in I})$ where $W$ is a group and $\{s_i\}_{i\in I}$  is a set of generators with defining relations
  $$(s_is_j)^{m_{ij}}=1\mbox{ for all }i,j\in I.$$
We call $\sM$ and $W$ {\em spherical} if $W$ is finite.
For $2\le k\le n$, we call $\sM$ and $W$ {\em $k$-spherical} if every subdiagram of $\sM$ induced on $k$ nodes is spherical.
\ede
\bde\label{dfn:cc loop}
Let $\sM$ be a Coxeter diagram over the index set $I=\{1,\ldots,n\}$.
The {\em Coxeter-Chein loop with diagram $\sM$}, denoted $C(\sM)$, is the Chein loop
 $M(W,2)$ where $(W,\{s_i\}_{i\in I})$ is a Coxeter system with diagram $\sM$.
\ede
From now on we shall fix a Coxeter system $(W,\{s_i\}_{i\in I})$ with spherical diagram $\sM$ over the index set $I=\{1,2,\ldots,n\}$.
We also assume that $u$ and $C=C(\sM)=M(W,2)$ are as in Definition~\ref{dfn:cc loop}.
 
\mn
Before we continue working with Coxeter groups, we shall first prove 
an open problem which was originally proposed by Petr Vojt\u{e}chovsk\'{y} in 2003.

\section{minimal presentations for $M(G,2)$}\label{section:involutions}

\bth\label{thm:involution generated groups}
Let $T$ be a group generated by a set $S=\{s_i\}_{i\in J}$ of elements. Here $J$ can be any set.
Then, $M(T,u)$ is the Moufang loop generated by $\{s_i\}_{i\in J}\cup \{u\}$ whose defining relations are those of $T$, together with the relations $((g_1g_2)u)^2=1$ where $g_1,g_2\in \{s_i\}_{i\in J}\cup\{e\}$.
\eth
The significance of this theorem is that the Chein relations are not a priori required to hold for {\em all} $g_1,g_2\in T$.

We shall prove Theorem~\ref{thm:involution generated groups} in a series of lemmas.
We shall assume that $L$ is a Moufang loop 
 containing $T$ and an element $u$ such that the relations $((g_1g_2)u)^2=1$ hold for $g_1,g_2\in \{s_i\}_{i\in J}\cup\{e\}$. Note that this is equivalent to saying that $u(g_1g_2)u=(g_1g_2)^{-1}$ for all $g_1,g_2\in \{s_i\}_{i\in J}\cup\{e\}$.

\ble\label{lem:chein relations hold}
The Chein relations (c1), (c2), and (c3) hold for any $g_1,g_2\in \{s_i\}_{i\in J}$.
\ele
\pf
We have
$$\begin{array}{rlc}
s_i(s_ju)
&= u(u(s_i(uus_ju))) & \\
&= u((us_iu)(us_ju)) & {\rm (m2)}\\
&=u(s_i^{-1}s_j^{-1}) & \\
&=uu(s_js_i)u & \\
&=(s_js_i)u, & \\
\end{array}$$
$$\begin{array}{rlc}
(s_iu)s_j
&= (((s_iu)s_j)u)u & \\
&= (s_i(us_ju))u & {\rm (m1)}\\
&= (s_is_j^{-1})u, & \\
\end{array}$$
and
$$\begin{array}{rlc}
(s_iu)(s_ju)
&= (us_i^{-1})(s_ju) & \\
&= u(s_i^{-1}s_j)u & {\rm (m3)}\\
&= s_j^{-1}s_i. & \\
\end{array}$$
\qed

\ble\label{lem:chein order 2}
For any $i,j\in J$ we have
 \begin{enumerate}
\item $s_iu=us_i^{-1}$
\item $s_i(s_ju)=(s_js_i)u$
\item $(s_iu)s_j=(s_is_j^{-1})u$
\item $(s_iu)(s_ju)=s_j^{-1}s_i$ 
\item $(us_i)s_j=s_j^{-1}(us_i)$
\end{enumerate}
\ele
\pf
(a) follows from (c3) taking $g_1=e$ and $g_2=s_i$. 
(b),(c), and (d) are just the Chein relations which hold from Lemma $\ref{lem:chein relations hold}$. 
(e) We use (a), (c), (d) and (a) again to see that
$(us_i)s_j=(s_i^{-1}u)s_j=(s_i^{-1}s_j^{-1})u=s_j^{-1}(s_i^{-1}u)=s_j^{-1}(us_i)$.
\qed
\ble\label{lem:weak c2}
For any $i_1,\ldots, i_k\in J$ we have 
 $u(s_{i_1}\cdots s_{i_k})=s_{i_1}^{-1}(u(s_{i_2}\cdots s_{i_k}))$ 
\ele
\pf
The following equalities are equivalent.
Write $w=s_{i_2}\cdots s_{i_k}$.
$$\begin{array}{rlc}
u(s_{i_1}  w)
&= s_{i_1}^{-1}(u  w) & \\
s_{i_1}w & = u(s_{i_1}^{-1}(u w)) & \mbox{ mult. by $u$}\\
s_{i_1}w & = ((u s_{i_1}^{-1}) u) w & \mbox{ (m2) }\\
s_{i_1} & = ((u s_{i_1}^{-1}) u)  & \mbox{ mult. by $w^{-1}$}\\
s_{i_1} & = s_{i_1} & \mbox{ Lemma~\ref{lem:chein order 2} (a)}\\
\end{array}$$

\qed
\ble\label{lem:conjugation by u}
For any $i_1,\ldots,i_k\in J$ we have
 $u(s_{i_1}\cdots s_{i_k})u=s_{i_k}^{-1}\cdots s_{i_1}^{-1}$.
In other words, for all $w\in T$, $uwu=w^{-1}$.
\ele
\pf
For $k=1$ this is true by (a) of Lemma~\ref{lem:chein order 2}.
For $k=2$ we have
$$\begin{array}{rlc}
u((s_{i_1}s_{i_2})u)
&= u(s_{i_2}(s_{i_1}u)) & {\rm (c1)}\\
&= u(s_{i_2}(us_{i_1}^{-1})) & {\rm induction}\\
&=((us_{i_2})u)s_{i_1}^{-1} & {\rm (m2)}\\
&=s_{i_2}^{-1}s_{i_1}^{-1} & {\rm induction}\\ 
\end{array}$$
Now suppose $k\ge 3$ and let $w=s_{i_2}\cdots s_{i_{k-1}}$. Then the following equalities are equivalent: 
$$\begin{array}{rlc}
u(s_1ws_k)u & = s_k^{-1}w^{-1}s_1^{-1}  & \\
u(s_1[u(s_k^{-1}w^{-1})u])u & = s_k^{-1}(u(s_1w)u)  & \mbox{ induction }\\
(u s_1)[(u(s_k^{-1}w^{-1}))u)u]  &= s_k^{-1}[(us_1)(w u)] &  \mbox{ (m3)}\\
s_k((u s_1)[u(s_k^{-1}w^{-1})]) &= (us_1)(w u) & \mbox{ mult. by $s_k$}\\
s_k((u s_1)[s_k(u w^{-1})])&= (us_1)(w u) &  \mbox{ Lemma~\ref{lem:weak c2}}\\
([s_k(u s_1)]s_k)(uw^{-1}) &=(us_1)(w u) &  \mbox{ (m2)}\\
 (u s_1)(uw^{-1}) &= (us_1)( w u) &  \mbox{ Lemma~\ref{lem:chein order 2} (e)}\\
uw^{-1} &= w u &  \mbox{ cancel $us_1$ }\\
wu &= w u &  \mbox{ induction }\\\end{array}$$
\qed
\ble\label{lem:c3}
For all $w_1,w_2\in T$ we have {\rm (c3)}:
 $(w_1u)(w_2 u)=w_2^{-1}w_1$.
\ele
\pf
We have 
$$\begin{array}{rlc}
(w_1u)(w_2 u) 
& = (uw_1^{-1})(w_2u) &\mbox{ by Lemma~\ref{lem:conjugation by u}}\\
& = u(w_1^{-1} w_2)u &\mbox{ (m3) }\\
& = w_2^{-1} w_1 &\mbox{ by Lemma~\ref{lem:conjugation by u}.}\\
\end{array}$$
 \qed
 \ble\label{lem:c1=c2}
 For any $w_1,w_2\in T$ we have
  {\rm (c1)} if and only if {\rm (c2)}.   
 \ele
\pf
The following are equivalent:
$$\begin{array}{rlc}
w_1(w_2 u) &= (w_2w_1)u & \mbox{ (c1)}\\
w_1(uw_2^{-1}) &= u (w_1^{-1} w_2^{-1}) & \mbox{ Lemma~\ref{lem:conjugation by u}}\\
(w_2u)w_1^{-1}&=  (w_2 w_1)u & \mbox{ taking inverses}\\
\end{array}$$
The latter equality is (c2).
\qed
\ble\label{lem:c2}
For all $w_1,w_2\in T$ we have {\rm (c2)};
 $(w_1u)w_2=(w_1w_2^{-1})u$.
\ele
\pf
The following are equivalent:
$$\begin{array}{rlc}
(w_1u) w_2  & = (w_1w_2^{-1})u & \mbox{ (c2) }\\
((w_1u) w_2)u  & = (w_1w_2^{-1}) & \mbox{ right mult. by $u$}\\
w_1(u (w_2u) ) & = (w_1w_2^{-1}) & \mbox{ (m2) }\\
w_1 w_2^{-1} & = (w_1w_2^{-1}) & \mbox{ by Lemma~\ref{lem:conjugation by u}.}\\
\end{array}$$
\qed

\pf (of Theorem~\ref{thm:involution generated groups})
By Lemmas~\ref{lem:c3},~\ref{lem:c1=c2},~and~\ref{lem:c2}, $L$ satisfies the Chein relations with respect to $u$ for all elements $g_1,g_2\in T$.
\qed

\bco\label{cor:involution-generated-groups}
Let $T$ be a group generated by a set $S=\{s_i\}_{i\in J}$ of elements of order two. Here $J$ can be any set.
Then, $M(T,u)$ is the Moufang loop generated by $\{s_i\}_{i\in J}\cup \{u\}$ whose defining relations are those of $T$, together with the Chein relations {\rm (c1)}, {\rm (c2)}, and {\rm (c3)} where $g_1,g_2\in \{s_i\}_{i\in J}\cup\{e\}$.
\eco

\bco\label{cor:coxeter chein relations}
Let $(W,\{s_i\}_{i\in I})$ be a Coxeter system with diagram $\sM$ over some set $I$.
Then, $M(W,u)$ is the Moufang loop generated by $\{s_i\}_{i\in I}\cup \{u\}$ whose defining relations are the Coxeter relations together with the Chein relations {\rm (c1)}, {\rm (c2)}, and {\rm (c3)} in which $g_1,g_2\in \{s_i\}_{i\in I}\cup\{e\}$.
\eco

\section{the automorphism group of $M(G,2)$}

In this section we get a better understanding of the automorphism groups of all the possible Chein loops $M(G,2)$.

\begin{lemma}\label{lem:trichotomy}
If $L=M(G,2)$ for some group $G$ then exactly one of the following holds:
\begin{enumerate}
\item\label{case 1} $L$ is an elementary abelian 2-group.
\item\label{case 2} $G \ncong M(H,2)$ for any group $H$.
\item\label{case 3} $G \cong M(H,2)$ for an abelian group $H$ where $H \ncong M(T,2)$ for any group $T$.
\end{enumerate}
\end{lemma}

\begin{proof}
If $G \cong M(H,2)$ for some group $H$ then, in order for $G$ to be a group, $H$ must be abelian.  If $H \cong M(T,2)=\left< T,x\right>$ for some group $T$ then, since $H$ is abelian, $x$ commutes with all of the elements in $T$. Thus $T$ is of exponent two.  Hence, $G$ and $L$ are elementary abelian 2-groups.
\end{proof}

\ble
If $L$ is a finite elementary abelian 2-group, then $\Aut(L)\cong GL_n(2)$ for some $n$.
\ele

\ble\label{lem:case 2 G char L}
Suppose $L=\left< G,u\right>\cong M(G,2)$ and $G \ncong M(H,2)$ for any group $H$.
Then, any $f\in \Aut(L)$ preserves both $G$ and $L\setminus G$.
\ele
\pf
Suppose $G'\ne G$ satisfies $L=G'\uplus G'u'\cong M(G',2)$ for some $u'\in L$.
Then since all elements of $G'u'$ satisfy the Chein relations with respect to the elements in $G'$ we may assume that $u'\in G$.
Let $H=G\cap G'$. Then, $G\cap G'u'=Gu'\cap G'u'=Hu'$, so 
 $G=H\uplus Hu'$. It follows that $G\cong M(H,2)$, a contradiction.
\qed

\bth\label{thm:case 2}  If $L=\left< G,u\right>\cong M(G,2)$ and $G \ncong M(H,2)$ for any group $H$ then $\Aut(L)\cong G\rtimes \Aut(G)$.
\eth

\begin{proof}
We first define the isomorphism 
\begin{alignat}{2}\label{eqn:Phi}
\Phi\colon G&\rtimes\Aut(G)& \longrightarrow \Aut(L)\\
&g\psi &\longmapsto \varphi_g\after \varphi_{\psi}.\nonumber
\end{alignat}
Here, for $g\in G$ define $\varphi_g : L \longrightarrow L$ by 
\begin{alignat}{3}\label{eqn:Phi G}
&\varphi_g(g_1)\: &=& \: g_1 \\
\textnormal{and } \:&\varphi_g(g_1 u)\: &=&\: (gg_1) u \nonumber
\end{alignat}
for any $g_1 \in G$. 
Moreover, for $\psi\in\Aut(G)$, we define its extension $\varphi_{\psi}$ to $L$ by setting 
\begin{alignat}{3}\label{eqn:Phi AutG}
&\varphi_{\psi}(g)\:&=\:&\psi(g)\\
\textnormal{and }\:&\varphi_{\psi}(gu)\: &=\: &\psi(g)u\nonumber
\end{alignat}
for any $g\in G$. 
It is evident that we then have 
\begin{alignat}{2}
\varphi_{\psi}\after\varphi_g\after\varphi_{\psi}^{-1} &= \varphi_{\psi(g)}\label{eqn:Aut MG2 conjugation relation}
\end{alignat}
In the sequel we shall simply write $\psi$ for $\varphi_\psi$.

We now prove $\Phi$ is a well defined isomorphism.
Clearly any $f\in \Aut(L)$ is uniquely determined by $f|_G$ and $f(u)$. Since by Lemma~\ref{lem:case 2 G char L} $G$ is characteristic in $L$, $f$ preserves $G$ and $Gu$.
Thus we have a natural homomorphism
 $\pi\colon \Aut(L)\to \Aut(G)$ given by $f\mapsto f|_G$.
We claim that $\Phi|_{\Aut(G)}$ is well defined and that $\pi\after\Phi|_{\Aut(G)}=\id$. 

Let $\psi\in \Aut(G)$ and denote $\Phi(\psi)$ also by $\psi$.
Since
\begin{alignat*}{1}
\psi((g_1u)(g_2u)) &= \psi(g_2^{-1}g_1)\\
             &= \psi(g_2)^{-1}\psi(g_1) \\
             &= \left(\psi(g_1)u\right)\left(\psi(g_2)u\right)\\
             &=\psi(g_1u)\psi(g_2u),
\end{alignat*}
\begin{alignat*}{1}
\psi(g_1(g_2u)) &= \psi((g_2g_1)u)\\
             &= \psi(g_2g_1)u \\
             &= \left(\psi(g_2)\psi(g_1)\right)u\\
             &= \psi(g_1)\left(\psi(g_2)u\right)\\
             &=\psi(g_1)\psi(g_2u),
\end{alignat*}
\begin{alignat*}{1}
\textnormal{and } \:\psi((g_1u)g_2) &= \psi((g_1g_2^{-1})u)\\
             &= \psi(g_1g_2^{-1})u \\
             &= \left(\psi(g_1)\psi(g_2)^{-1}\right)u\\
             &= \left(\psi(g_1)u\right)\psi(g_2)\\
             &=\psi(g_1u)\psi(g_2)
\end{alignat*}
for all $g_1,g_2 \in G$, and $\psi^{-1}(gu)=\psi^{-1}(g)u$ for any $g\in G$,  $\psi \in \Aut(L)$.  
This proves the claim on $\Phi|_{\Aut(G)}$.
It follows that  
\begin{align*}
K=\Phi(\Aut(G))\cong \Aut(G)
\end{align*}
 is a complement to $N=\ker\pi$, so that $\Aut(L)\cong N\rtimes K$.

We now examine $N$.
We define a map $\chi\colon N\to G$ by $f(u)=\chi(f)u$.
This is a homomorphism because
\begin{align}
(f\after g)(u)  = f(\chi(g)u)=\chi(g)f(u)=\chi(g)(\chi(f)u)=(\chi(f)\chi(g))u. \label{eqn:chi hom}
\end{align}
Clearly any  $f\in \ker(\chi)\le N$ fixes $u$ as well as all elements in $G$, so $\chi$ is injective.
We claim that $\Phi|_G$ is well-defined and $\chi\after\Phi|_G=\id$. Let $g\in G$ and write $\Phi(g)=\varphi_g$.
Since
\begin{alignat*}{1}
\varphi_g((g_1u)(g_2u)) &= \varphi_g(g_2^{-1}g_1)\\
             &= g_2^{-1}g^{-1} gg_1 \\
             &= \left((gg_1)u\right)\left((gg_2)u\right)\\
             &=\varphi_g(g_1u)\varphi_g(g_2u),
\end{alignat*}
\begin{alignat*}{1}
\varphi_g(g_1(g_2u)) &= \varphi_g((g_2g_1)u)\\
             &= (gg_2g_1)u \\
             &= g_1\left((gg_2)u\right)\\
             &=\varphi_g(g_1)\varphi_g(g_2u),
\end{alignat*}
\begin{alignat*}{1}
\textnormal{and } \:\varphi_g((g_1u)g_2) &= \varphi_g((g_1g_2^{-1})u)\\
             &= (gg_1g_2^{-1})u \\
             &= \left((gg_1)u\right)g_2\\
             &=\varphi_g(g_1u)\varphi_g(g_2)
\end{alignat*}
for all $g_1,g_2 \in G$, and $\varphi_g^{-1}=\varphi_{g^{-1}}$, $ \varphi_g \in \Aut(L)$.   Clearly $\chi\after\Phi(g)=g$. This proves the claim on $\Phi|_G$.
Thus 
\begin{align*}
N=\Phi(G)=\left\{\varphi_g \: |\: g\in G \right\}\cong G.
\end{align*}
We conclude that $\Aut(L)\cong N\rtimes K\cong G\rtimes\Aut(G)$.
\end{proof}

\begin{theorem}\label{thm:autG2}  If $L=\left< G,u_2\right>\cong M(G,2)$ for a nonabelian group $G = \left<H,u_1\right>\cong M(H,2)$ then $\Aut(L)\cong (H\times H)\rtimes (S_3 \times \Aut(H))$.
\end{theorem}

\begin{proof}
First notice that $L=H\uplus Hu_1\uplus Gu_2=H\uplus Hu_1\uplus Hu_2\uplus Hu_3$, where $\{1,u_1,u_2,u_3\}$ forms a Klein four group.
Clearly any $f\in \Aut(L)$ is uniquely determined by its action on $H$, $u_1$ and $u_2$.
Since $H$ is abelian, but $H\ncong M(T,2)$ for any group $T$, $L$ contains elements of order $\ge 2$ and they are all in $H$.
Let $h$ be any such element, then $H=C_L(h)$, so $H$ is characteristic in $L$ and $f|_H\in \Aut(H)$.
Thus $f$ fixes $H$, with restriction in $\Aut(H)$, while permuting the remaining three cosets of $H$.
Thus we have a homomorphism $\pi\colon\Aut(L)\to S_3\times \Aut(H)$. We now show that $\pi$ is surjective and that $N=\ker\pi$ has a complement $K\cong S_3\times \Aut(H)$ so that 
 $\Aut(L)\cong N\rtimes K$.

For $i=1,2$, let $\Psi_i\colon H\rtimes \Aut(H)\to \Aut(G_i)$ be the injective homomorphism defined in~\eqref{eqn:Phi},  viewing $G_i\cong M(H,2)$.
Likewise, for $i=1,2$, let $G_i=\langle H,u_i\rangle\cong M(H,2)\cong G$ and let $\Phi_i\colon G_i\rtimes \Aut(G_i)\to \Aut(L)$ be as in~\eqref{eqn:Phi}, viewing $L=M(G_i,2)=\langle G_i,u_{3-i}\rangle$.

Let 
\begin{align*}
A=(\Phi_1\after\Psi_1)(\Aut(H)).
\end{align*}
To see what this group is, let $\psi\in \Aut(H)$ and denote $\psi=\Phi_1\after\Psi_1(\psi)$.
Then, $\psi(hu_i)=\psi(h) u_i$, for $i=1,2,3$ and all $h\in H$.

For $i=1,2$, set $\sigma_i=\Phi_i(u_i)\in \Aut(L)$; recall $u_i\in G_i$. Then, $\sigma_i$ fixes $G_i=H\uplus Hu_i$ elementwise, while, for any $h\in H$ we have
 \begin{alignat*}{3}
 hu_{3-i}&\mapsto (u_ih)u_{3-i}&=h(u_iu_{3-i})&=hu_3\\
 hu_3&\mapsto (u_ih)(u_iu_{3-i})&=(u_i^2 h)u_{3-i}&=hu_{3-i}.
\end{alignat*} 
Thus $S=\langle \sigma_1,\sigma_2\rangle\cong S_3$ is a subgroup of $\Aut(L)$ mapping onto $S_3\times\{1\}$ under $\pi$. 
Clearly any $\psi\in A$ commutes with any element of $S$.
Thus 
\begin{align*}
K=\langle S, A\rangle \cong S_3\times\Aut(H)
\end{align*}
 is a complement to $N=\ker\pi$.

We now show that $N\cong H\times H$.
Define a map $\chi\colon N \to H\times H$ by sending 
 $f\mapsto (\chi_1(f),\chi_2(f))$, where $f(u_i)=\chi_i(f)u_i$ for $i=1,2$.
As in~\eqref{eqn:chi hom}, this is a homomorphism.
Clearly $\chi$ is injective because $f\in \ker\chi\le N$ fixes each element of $H$, as well as  each $u_i$ for $i=1,2,3$.
Finally we notice that $\chi$ is onto.

Fix $i\in \{1,2\}$ and $h\in H$.
Then, for $h'\in H$ we have $\Psi_i(h)(h')=h'$ and $\Psi_i(h)(h'u_i)=(hh')u_i$.
Moreover, $(\Phi_i\after\Psi_i)(h)$, extends $\Psi_i(h)\in \Aut(G_i)$ to $L$ as $g u_{3-i}\mapsto (\Psi_i(h)(g))u_{3-i}$, for all $g\in G_i$.
Summarizing, $\Psi_i(h)$ is given by 
\begin{alignat}{3}
h'\mapsto h' &\hspace{2em}(h'u_i)\mapsto (hh')u_i &\hspace{1em}h'u_{3-i}\mapsto h'u_{3-i} &\hspace{2em}
 h'u_3 \mapsto (h'h^{-1})u_3
\end{alignat}
for any $h'\in H$.
Hence, if, for $h_1,h_2\in H$ we set $f=((\Phi_1\after\Psi_1)(h_1))\after ((\Phi_2\after\Psi_2)(h_2))\in \Aut(L)$, 
then, $f(u_i)=h_iu_i$, so $(\chi_1(f),\chi_2(f))=(h_1,h_2)$.
We conclude that $N\cong H\times H$ and 
 $\Aut(L)\cong (H\times H)\rtimes (S\times \Aut(H))$.
\end{proof}

\section{amalgams}
Let $(W,\{s_i\}_{i\in I})$ be a Coxeter system with diagram $\sM$ over $I$ and let $L=M(W,2)$ be the associated Coxeter-Chein loop.
Let $\bI=I\cup \{\infty\}$ and set $s_\infty=u$.

The Coxeter-Chein presentation can be interpreted as a simplicial amalgam in the sense of~\cite{BloHof2011}. We first construct a simplicial complex from the graph underlying $\sM$.

\bde\label{dfn:underlying graph}
Let $\sM$ be a Coxeter diagram over $I$ associated with Coxeter matrix $(m_{ij})_{i,j\in I}$.
We call $\De(\sM)=(I,E)$ the {\em graph underlying} $\sM$, where
 $E=\{\{i,j\}\in I\times I\mid m_{ij}\ge 3\}$.
 \ede 

\noindent
From now on let $\De=(I,E)$ be the graph underlying $\sM$.

\bde\label{dfn:simplicial complex}
A {\em simplicial complex} is a pair $\bE=(E,F)$, where 
 $E$ is a set and $F$ is a collection of subsets of $E$, called {\em simplices},  such that 
  $\{e\}\in F$ for each $e\in E$,
 and if $\tau\in F$ and $\rho$ is a subset of $\tau$, then $\rho\in F$; in this case we say that $\rho$ is a {\em face} of $\tau$ and we write $\rho\le \tau$.
A simplex $\sigma$ with $|\sigma|=r+1$ is called an {\em $r$-simplex}.\ede

\bde Given a graph $\De=(I,E)$ we define its associated simplicial complex as
 $\bE(\De)=(E,F)$, where $F$ is the collection of subsets of cardinality at most $3$ of $E$.
\ede
\noindent We now let $\bE=(E,F)$ be the simplicial complex associated to $\De$.

\bde\label{dfn:simplicial amalgam}
A {\em simplicial amalgam of Moufang loops} over a simplicial complex $\bE=(E,F)$ is a collection of Moufang loops with connecting homomorphisms, $$\cG=\{G_\sigma,\psi_\tau^\rho\mid \rho,\sigma,\tau\in F, \rho\le \tau\},$$ 
where, for all $\rho,\tau\in F$ with $\rho\le\tau$, we have an injective homomorphism $$
\psi_\tau^\rho\colon G_\tau\into G_\rho.$$ 
We require that, whenever $\rho\le \sigma\le \tau$, then
 $\psi_\tau^\rho=\psi_\sigma^\rho\after\psi_\tau^\sigma$.

A {\em completion} of $\cG$ is a Moufang loop $G$ together with a
collection  $\phi=\{\phi_\sigma\mid \sigma\in F\}$ of
homomorphisms $\phi_\sigma\colon G_\sigma\to G$, such that whenever
$\sigma\le \tau$, we have
$\phi_\sigma\after\psi^\sigma_\tau=\phi_\tau$. The amalgam
$\cG$ is {\em non-collapsing} if it has a non-trivial completion. A
completion $(\hat{G},\hat{\phi})$ is called {\em universal} if for
any completion $(G,\phi)$ there is a  (necessarily unique)
surjective loop homomorphism $\pi\colon \hat{G}\to G$ such that
$\phi=\pi\after\hat{\phi}$. 
\ede
\bre
Definition~\ref{dfn:simplicial amalgam} is a specialization of the simplicial amalgams in concrete categories defined in~\cite{BloHof2011}
\ere

We shall now define a simplicial amalgam $\cG$ of Moufang loops as in Definition~\ref{dfn:simplicial amalgam} over the simplicial complex  $\bE=(E,F)$.
For any subset $J\sbe I$, we let $$L_J=\langle s_j,s_\infty\mid j\in J\rangle_L\cong M(W_J,2).$$
Then, for $\sigma\in F$ we let $G_\sigma$ be a copy of $L_J$, where
 $J=\bigcap_{e\in \sigma}e$.
Moreover, for all $\rho,\tau\in F$ with $\rho\le\tau$, $\psi_\tau^\rho\colon G_\tau\into G_\rho$ is given by natural inclusion of subloops in $L$.

 \ble $L$ is a completion of $\cG$.
 \ele
 \pf
In Definition~\ref{dfn:simplicial amalgam}
 take the collection $\phi=\{\phi_\sigma=\id\mid \sigma\in F\}$.
\qed

\bde
An amalgam {\em of type} $\cG$ is an amalgam 
 $\cG'=\{G_\sigma,\varphi^\rho_\tau\mid \rho,\sigma,\tau\in F,\rho\le \tau\}$, such that, for each $\sigma\in F$, the loop $G_\sigma$ is that of $\cG$, and, for each $\rho,\tau\in F$ with $\rho\le \tau$, the images of the connecting maps $\varphi^\rho_\tau$ and $\psi^\rho_\tau$ coincide.
 \ede
Our aim is now to classify all amalgams of Moufang loops of type
 $\cG$ using $1$-cohomology on $\bE$ as described in~\cite{BloHof2011}.
\bde\label{dfn:coefficient system}
Let $\cG=(G_\bullet,\psi_\bullet)$ be an amalgam over the simplicial complex $\bE=(E,F)$. The {\em coefficient system associated to $\cG$} is the collection
$$\cA=\{A_\sigma, \alpha^\sigma_\rho\mid \sigma< \rho\mbox{ with }\sigma,\rho\in F \},$$
where, for each $\sigma\in F$, we let 

$$A_\sigma=\bigcap_{\rho>\sigma}\Stab_{\Aut(G_\sigma)}(G_\rho),$$
and, for each $\sigma,\rho\in F$ with $\sigma< \rho$,  we have a homomorphism
$$\begin{array}{lrl}
&\alpha^\sigma_\rho\colon A_\sigma&\to A_\rho\\
&f&\mapsto f|_{G_\rho}\\
\end{array}$$
If $\sigma=\{e_1,\ldots,e_l\}$ we shall sometimes write 
 $A_{e_1,\ldots,e_l}$ for $A_\sigma$.
As for amalgams, we shall use the shorthand notation
$\cA=\{A_\bullet,\alpha_\bullet\}$.
\ede
\bre
Strictly speaking, one should define $\alpha^\sigma_\rho(f)$ as $(\psi^\sigma_\rho)^{-1}\after f\after \psi^\sigma_\rho$, where, $(\psi^\sigma_\rho)^{-1}$ is understood to be defined only on 
 the image of $\psi^\sigma_\rho$.
However, since the map $\psi^\sigma_\rho$ is given by the inclusion $G_\rho\sbe G_\sigma$ of subloops of $L$, we have chosen to describe them as above for better readability.
\ere

\noindent From now on  the coefficient system associated to $\cG$ defined above will be $$\cA=\{A_\sigma, \alpha^\sigma_\rho\mid \sigma< \rho\mbox{ with }\sigma,\rho\in F \}.$$

We now compute the groups and connecting maps of $\cA$.
\ble\label{lem:Asigma}
Let $e,f,g\in E$ be distinct and let $\sigma\in F$. Then, 
\begin{enumerate}
\renewcommand{\theenumi}{{\rm (\alph{enumi})}}\item If $\bigcap_{e\in\sigma}e=\emptyset$, then $A_\sigma=\{\id\}$.
\item If $e\cap f=\{j\}$, then 
$A_{e,f}=\langle \gamma_j\rangle\cong \ZZ_2$, where $\gamma_j\colon s_j\leftrightarrow s_js_\infty$;
\item If $e=\{i,j\}$ with $m_{ij}\ge 3$, then $A_e=\langle \gamma_{ij}\rangle\cong\ZZ_2$, where 
 $\gamma_{ij}\colon\left\{\begin{array}{@{}l} 
 s_\infty\mbox{ is fixed}\\
 s_i\leftrightarrow s_is_\infty\\
 s_j\leftrightarrow s_js_\infty\\
 s_is_j\leftrightarrow s_js_i\\
 \end{array}\right.$.\\
 Thus, if $f=\{j,k\}$, then 
 $\alpha_{e,f}^e\colon \gamma_{ij}\mapsto \gamma_j$;
 \item If $e\cap f\cap g=\{j\}$, then $A_{e,f,g}=\langle \gamma_j\rangle\cong \ZZ_2$, where $\gamma_j\colon s_j\leftrightarrow s_js_\infty$.
 In this case $\alpha_{e,f,g}^{e,f}\colon \gamma_j\mapsto \gamma_j$.
 \end{enumerate}
 \ele
\pf
(a)
This is clear since $G_\sigma$ has only one non-trivial element.

(b)
We have $G_{\{i,j\},\{j,k\}}=L_j\cong 2^2$.
Thus $\Aut(L_j)=\Sym(\{s_j,s_\infty,s_js_\infty\})$.
The elements of $A_{\{i,j\},\{j,k\}}$ also preserve
 $G_{\{i,j\},\{j,k\},\{i,k\}}=L_\emptyset=\langle s_\infty\rangle_L$ and so they must permute $\{s_j,s_js_\infty\}$. Call the non-trivial automorphism $\gamma_j$.

(c)
Assume $m_{ij}\ge 3$. Let $\gamma$ be a non-trivial element in $A_{\{i,j\}}$. Then $\gamma$ is an automorphism of $L_{ij}$ preserving $L_i$, $L_j$ and $L_\emptyset$.
Setting $G=W_{\{i,j\}}$ and $H=\langle s_is_j\rangle$ it follows from 
 Theorem~\ref{thm:autG2} that $H$ is characteristic in $L$.
Thus any automorphism of $L_{ij}$ permutes the three copies of $G$ in $L_{ij}$:
 $C_i=\langle H, s_i\rangle=\langle H, s_j\rangle$, 
  $C_\infty=\langle  H,s_\infty\rangle$, and $C_{i,\infty}=\langle H,s_is_\infty\rangle=\langle H,s_js_\infty\rangle$.
The automorphism $\gamma$ therefore simultaneously permutes the intersections:
 $C_i\cap L_i=\langle s_i\rangle$, $C_\infty\cap L_i=\langle s_\infty\rangle$, and $C_{i,\infty}\cap L_i=\langle s_is_\infty\rangle$
as well as those with $L_j$.
Thus, since $s_\infty$ is fixed, $\gamma$ must interchange $s_i$ and $s_is_\infty$ as well as $s_j$ and $s_js_\infty$.
It then follows that 
 $s_is_j\mapsto (s_is_\infty) (s_js_\infty)=(s_\infty s_i)(s_js_\infty)=s_\infty(s_is_j)s_\infty=s_js_i$, so $\gamma$ acts by sending every element of $H$ to its inverse.
(d) Immediate from (b).
\qed

\bre\label{rem:F2 vector space}
Lemma~\ref{lem:Asigma} can be summarized as saying that 
$$A_\sigma\cong
\begin{cases}
\FF_2 & \text{ if } \bigcap_{e\in \sigma}\sigma\ne \emptyset\\
\{0\}  & \text{ else }\\
\end{cases}$$ 
and, for $\sigma<\rho$, $\alpha^\sigma_\rho$ is 
 an isomorphism if and only if  $\bigcap_{e\in \rho}e\ne \emptyset$
or $\bigcap_{e\in \sigma}e=\emptyset$. 
\ere
\noindent 
It follows from Remark~\ref{rem:F2 vector space}  that the cochain complex of pointed sets obtained from the coefficient system $\cA=\{A_\bullet,\alpha_{\bullet}\}$ in~\cite{BloHof2011} can be constructed directly from the underlying graph $\De$ of $\sM$. 

\bde\label{dfn:cochain}
Given a graph $\De=(I,E)$ with associated simplicial complex $\bE=(E,F)$, 
we define $F_\bullet^r=\{\sigma\in F^r\mid \bigcap_{e\in \sigma}e\ne\emptyset\}$ and $F_\bullet=\bigcup_{r\in \NN}F^r_\bullet$. 
We then define 
 a cochain complex $\cC^\bullet(\De)$ of $\FF_2$-vector spaces:
$$\cC^0\stackrel{d^0}{\longrightarrow}\cC^1 \stackrel{d^1}{\longrightarrow}\cC^2 $$
where $\cC^i$ is an $\FF_2$-vector space with formal basis $\{a_\sigma\mid \sigma\in F^r_\bullet\}$ and, for $r=0,1$, the coboundary map $d^r$ is $\FF_2$-linear and given by
 $$d^r\colon a_\sigma\mapsto \sum_{\tau\in F^{r+1}_\bullet\colon \sigma<\tau}a_\tau.$$
\ede
\ble
The cochain complex of pointed sets associated to the coefficient system $\cA=\{A_\bullet,\alpha_{\bullet}\}$ as defined in~\cite{BloHof2011} is $\cC^\bullet(\De)$.
\ele

\noindent From now on let $\cC^\bullet=\cC^\bullet(\De)$. 

\bde
\label{dfn:cohomology}  
Given a graph $\De$ and $r=0,1$,  we define
$$\begin{array}{rlr@{}l}
Z^r(\De)&=\ker d^r & r&\mbox{-cocycle group}\\
B^{r+1}(\De)&=\im d^r & (r+1)&\mbox{-coboundary group}\\
H^r(\De)&=Z^r(\De)/B^r(\De) & r&\mbox{-cohomology group}\\
\end{array}$$
We denote the cohomology class of $z\in Z^r(\De)$ as $[z]=z+B^r(\De)$. Here $d^r$ is the $r$-coboundary map of the cochain complex $\cC^\bullet(\De)$.
\ede
The motivation to compute these groups is the following.
\bth\label{thm:cohomology correspondence}{\rm (cf. Theorem 2 of~\cite{BloHof2011})}
The isomorphism classes of amalgams of type $\cG$ are bijectively parametrized by elements of $H^1(\De)$.
\eth
To make this parametrization explicit, we need a concrete basis for $H^1(\De)$.
\paragraph{A basis for $H^1(\De)$}

We consider the cochain complex restricted to the set of edges on a vertex.

\bde Let $\De=(I,E)$ be a graph and let $i\in I$.
Let $\De_i=(I_i,E_i)$ be the subgraph on the set of edges incident to $i$. Write $v(i)=|E_i|$.
\ede
\noindent For any $i\in I$, let $\bE_i=(E_i,F_i)$ denote the simplicial complex $\bE(\De_i)$. Since $i\in \bigcap_{e\in \sigma} e$ for all $\sigma\in F_i$, $\De_i$ is the full simplex on $E_i$.

\bde\label{dfn:vertex decomposition}
Given a graph $\De=(I,E)$ and $i\in I$, define $\cC_i^\bullet(\De)$ to be  the cochain subcomplex of $\cC^\bullet(\De)$ associated to $\De_i$ :
$$\cC^0_i\stackrel{d^0_i}{\longrightarrow}\cC^1_i \stackrel{d^1_i}{\longrightarrow}\cC^2_i $$
That is, for each $r=0,1,2$, $\cC^r_i$ is the subspace of $\cC^r$ spanned by $\{a_\sigma\mid \sigma\in F_i^r\}$ and coboundary maps are given by:
$$\begin{array}{lr}
\begin{array}{rl}
d^0_i  \colon\cC_i^0&\to \cC_i^1\\
a_\sigma&\mapsto \sum_{\tau\in F^{1}_i\colon \sigma<\tau}a_\tau.
\end{array}
&\mbox{ and }
\begin{array}{rl}
d^1_i \colon \cC_i^1& \to \cC_i^2,\\
a_\sigma&\mapsto \sum_{\tau\in F^{2}_i\colon \sigma<\tau}a_\tau.
\end{array}
\end{array}$$
Note that $\cC^\bullet_i(\De)$ is naturally isomorphic to $\cC^\bullet(\De_i)$.
\ede
\ble\label{lem:vertex decomposition}
\ \\
\begin{enumerate}
\item 
$$\begin{array}{rll}
\cC^1&=\bigoplus_{i\in I} \cC^1_i&=\bigoplus_{i\in I\colon v(i)\ge 2} \cC^1_i,\\
\cC^2&=\bigoplus_{i\in I} \cC^2_i&=\bigoplus_{i\in I\colon v(i)\ge 3} \cC^2_i.\\
\end{array}$$
\item For $r=0,1$ and $i\in I$, we have  
$$\begin{array}{lll}
 d^r_i\colon \cC_i^r=\{0\}& \to \cC_i^{r+1}=\{0\}&\text{ if } v(i)\le r\\
 d^r_i\colon \cC_i^r=\langle a_{E_i}\rangle& \to\cC^{r+1}_i=\{0\} & \text{ if }v(i)=r+1\\
\end{array}$$
\item
$d^1=\sum_{i\in I\colon v(i)\ge 2} d^1_i$.
\item 
$Z^1(\De)=\bigoplus_{i\in I\colon v(i)=2}\cC^1_i\oplus \bigoplus_{i\in I\colon v(i)\ge 3}Z^1(\De_i)$.
\end{enumerate}
\ele
\pf
Immediate from Definitions~\ref{dfn:cochain}
~and~\ref{dfn:vertex decomposition}.
\qed

\ble\label{lem:vertex bases} 
Suppose $\De=(I,E)$ is a star with
 $E=\{e_0,\ldots,e_{v-1}\}$.
\begin{enumerate}
\renewcommand{\theenumi}{{\rm (\alph{enumi})}}
\item $Z=\{d^0(a_{e_i})\mid 0< i\le v-1\}$ is a basis for $B^1(\De)=Z^1(\De)$.
\item $B=\{d^1(a_{e_i,e_j})\mid 0< i<j\le v-1\}$ is a basis for $B^2(\De)$.
\end{enumerate}
\ele
\pf
(a)
That $Z$ is linearly independent follows from the fact that  $d^0(a_{e_i})$ is the only element in $Z$ with non-zero coefficient for $a_{e_i,e_0}$. 
Now note that  
 $$\begin{array}{rl}
 \sum_{i=0}^{v-1}d^0(a_{e_i})&=\sum_{i=0}^{v-1}\sum_{j\ne i} a_{e_i,e_j}\\
 &=\sum_{\{i,j\}\sbe \{0,1,\ldots,v-1\}} 2 a_{e_i,e_j}=0.\\
 \end{array}$$
Hence, $d^0(a_{\{0\}})\in\langle Z\rangle$.
That $Z$ is also a basis for $\ker d^1$ will be proved below.
(b)
We compute $\dim\ker d^1=\dim\cC^1-\dim\im d^1$.
Clearly $\dim\cC^1=\binom{v}{2}$.
If $v=1$, then $\dim\cC^1=0=\dim\ker d^1=v-1$.
If $v=2$, then clearly $\dim\im d^1=\dim \cC^2=0$ and so $\dim \ker d^1=1-0=v-1$.
If $v=3$, then $\dim\im d^1=\dim \cC^2=1$ and so $\dim \ker d^1=3-1=v-1$.
Suppose now that $v\ge 4$. Then, the set 
 $B=\{d^1(a_{e_i,e_j}) \mid 1\le i<j\le v-1\}$ is linearly independent because $d^1(a_{e_i,e_j})$ is the unique vector in $B$ with non-zero coefficient for $a_{e_0,e_i,e_j}$.
So $\dim\im d^1\ge \binom{v-1}{2}$.
Note however that for any $i\in \{1,\ldots,v-1\}$, 
 $$\begin{array}{rl}
 \sum_{j\ne i}d^1(a_{e_i,e_j})&=\sum_{j\ne i}\sum_{k\ne i,j} a_{e_i,e_j,e_k}\\
 &=\sum_{\{j,k\}\sbe \{0,\ldots,\hat{\imath},\ldots,v-1\}} 2a_{e_i,e_j,e_k}=0.\\
 \end{array}$$
Hence  $d^1(a_{e_i,e_0})=\sum_{j\ne i,0} d^1(a_{e_i,e_j})\in \langle Z\rangle$.
Thus, $\dim \im d^1\le \binom{v-1}{2}$.
It follows that 
 $$\dim \ker d^1=\binom{v}{2}-\binom{v-1}{2}=v-1.$$ 
To see that $Z$ is a basis for $\ker d^1$ note that
$d^1\after d^0=0$ since for each $i\in V$ we have
 $d^1\after d^0(a_{e_i})=\sum_{j\ne i}d^1(a_{e_i,e_j})$ and this was shown to equal $0$.
\qed

\ble\label{lem:dim Z1}\label{lem:dim B1}
\begin{enumerate}
\renewcommand{\theenumi}{(\alph{enumi})}
\item 
We have $\dim_{\FF_2}(Z^1(\De))=\sum_{i\in I}(v(i)-1)=2|E|-|I|$.
More precisely, for each $i\in I$, pick any $f_i\in E_i$. Then $Z^1(\De)$ has basis $$Z_\De=\{d^0_i(a_e)\mid i\in I\mbox{ with }v(i)\ge 2 \mbox{ and }e\in E_i-\{f_i\}\}.$$
\item
We have $\dim_{\FF_2} B^1(\De)=|E|-1$.
More precisely, given any $e_0\in E$, it has $\FF_2$-basis  
$$B_\De=\left\{ d^0(a_e)\mid e\in E-\{e_0\}\right\}.$$
\end{enumerate}
\ele
\pf
(a) 
This follows by combining Lemma~\ref{lem:vertex decomposition} part (d) and Lemma~\ref{lem:vertex bases}.

(b)
Suppose that we have
 \beq\label{eqn:zero comb}\sum_{e\in E}\lambda_e d^0(a_e)=0 \mbox{ with }\lambda_{e_0}=0.\eeq
From Lemma~\ref{lem:vertex bases} we derive the following observation: For each $i\in I$ and $f_i\in E_i$, 
\begin{eqnarray}
&\text{ the set $\{d^0_i(a_e)\mid e\in E_i-\{f_i\}\}$ is linearly independent in $Z^1(\De)$,}&\label{indep}\\
&\text{$d^0_i(a_{f_i})=\sum_{e\in E_i-\{f_i\}}d^0_i(a_e)$}& \label{sum}.
\end{eqnarray}
Now it follows from (\ref{eqn:zero comb}),~\ref{indep}~and~\ref{sum}, that 
 if for some $i\in I$ and $f_i\in E_i$ we have  
 $\lambda_{f_i}=0$, then also $\lambda_e=0$ for all $e\in E_i-\{f_i\}$.
Since $\De$ is connected and we have $\lambda_{e_0}=0$, it follows that $\lambda_e=0$ for all $e\in E$.
Also note that 
 $$\sum_{e\in E}d^0(a_e)=\sum_{i\in I\colon v(i)\ge 2}\sum_{e\in E_i} d^0_i(a_e)=0,$$ so $d^0(a_{e_0})\in \langle d^0(a_e)\mid e\in E-\{e_0\}\rangle$.
Thus, $B_\De$ is linearly independent and spans $B^1(\De)$.
\qed

\mn
Let $\sT$ be a spanning tree for $\De$ with edge set $E(\sT)=E-\{e_1,\ldots,e_n\}$ and, for $j=1,\ldots,n$, let $e_j=\{o_j,t_j\}$.
\bpr\label{prop:dim H1}
We have 
$$\dim_{\FF_2}H^1(\De)=|E|-|I|+1=n.$$
More precisely, $H^1(\De)$ has a basis
$$H_\De=\{[d^0_{o_j}(e_j)]\mid j=1,2\ldots,n\}.$$
\epr
\pf
From Lemma~\ref{lem:dim Z1}, 
we find 
 $$\begin{array}{rl}
 \dim_{\FF_2}H^1(\De)&=\dim_{\FF_2}Z^1(\De)-\dim_{\FF_2}B^1(\De)\\
 &=2|E|-|I|-(|E|-1)\\&=|E|-|I|+1.
 \end{array}$$
Recalling that $\sT$ is a tree on the vertex set $I$, we find $|E|-|I|+1=(|E(\sT)|-|I|+1)+n=0+n$.
To describe a basis for $H^1(\De)$, define a tree $$\bar{\sT}=(I\cup\{\bar{t_j},\bar{o_j}\},E(\sT)\cup \{\bar{e_j},\bar{e_j}'\}_{j=1}^n),$$ where
  $\bar{e_j}=\{o_j,\bar{t_j}\}$ and $\bar{e_j}'=\{\bar{o_j},t_j\}$ for each $j=1,2,\ldots,n$ (we double each edge $e_j$ and attach one copy to each of its original endpoints).
The map $\kappa\colon\bar{\sT}\to \De$ given by 
$$\kappa(i) =\begin{cases}
  i & \text{if } i\in I\\
t_j & \text{if } i=\bar{t_j}\\
o_j & \text{if } i=\bar{o_j}\\
\end{cases}$$  
induces an isomorphism $\kappa\colon Z^1(\bar{\sT})\to Z^1(\De)$ since valency of the vertices in $I$ is preserved and the valency of $\bar{t_j}$ and $\bar{o_j}$ is $1$.
On the other hand $\kappa\colon B^1(\bar{\sT})\to B^1(\De)$ sends both $d^0(\bar{e_j})=d_{o_j}^0(\bar{e_j})$ and 
  $d^0(\bar{e_j}')=d_{t_j}^0(\bar{e_j}')$ to 
  $d^0(e_j)=d_{o_j}^0(e_j)+d_{t_j}^0(e_j)$.
So, $H_\De$ is a basis for $H^1(\De)=Z^1(\De)/B^1(\De)$.
\qed

\paragraph{Description of the amalgams}

\bth{\rm(Classification of Loops of type $\cG$)}
Let $\sT=(I,E(\sT))$ be a spanning tree for $\De$ and let 
  $E-E(\sT)=\{e_1,\ldots,e_n\}$. For each $j\in \{1,\ldots,n\}$ pick a vertex $o_j\in e_j$.
There is a bijection between amalgams 
 $$\begin{array}{rl}
 \cL^\delta&=\{L_i,L_e,\varphi_i^e\mid i\in I, i\in e\in E,\varphi_i^e\colon L_{i}\into L_{e}\}\\
 \end{array}$$ 
of type $\cG$  and subsets $\delta\sbe \{1,2,\ldots,n\}$ given as follows:
For every vertex $i\in I$ and edge $e$ with $i\in e$, we have
 \beq\label{eqn:amalgam parametrization}\varphi_i^e=
 \begin{cases}
 \begin{array}{@{}ll}
s_i & \mapsto s_i s_\infty\\
s_\infty & \mapsto s_\infty \\
\end{array}
&\text{ if }i=o_j\mbox{ and } e=e_j\mbox{ for some }j\in \delta\\
\psi_i^e=\id & \mbox{ else }\\ 
 \end{cases}
 \eeq
In particular, for $\delta=\emptyset$, we recover $\cG$.
\eth

\pf We shall use the parametrization alluded to in Theorem~\ref{thm:cohomology correspondence}.
Using Proposition~\ref{prop:dim H1} we see that the subsets $\delta\sbe\{1,2,\ldots,n\}$ bijectively parametrize the elements in the $\FF_2$-vector space $H^1(\De)$ as follows:
$$\delta\mapsto \left[z=\sum_{j\in \delta} d^0_{o_j}(e_j)\right].$$
Following~\cite[\S 4]{BloHof2011}, we shall use the cocycle [z] to construct the corresponding normalized amalgam
 $$\cG^{[z]}=\{G_\sigma,\varphi^\rho_\tau\mid \rho,\sigma,\tau\in F,\rho\le \tau\}.$$
Choose a total order $\prec$ on $E$ such that $e_j\prec f$ for all $f\in E_{o_j}\cap E(\sT)$ (we can even achieve $e_j\prec f$ for all $f\in E(\sT)$).
For $\rho,\tau\in F$ with $\rho\le \tau$, we define $\varphi_\tau^\rho$ as follows (here $e=\max_\prec \rho$ and $f=\max_\prec\tau$ so $e\preceq f$):
\beq\label{eqn:cycle to inclusions}
\varphi_\tau^\rho=
\begin{cases}
\psi_\tau^\rho &\text{ if }e=f\\
(\psi_\rho^{\{e\}})^{-1}\after\psi_{\{e,f\}}^{\{e\}} z_{\{e,f\}}^{-1}\after \psi_\tau^{\{e,f\}}&\text{ if }e\prec f,\\      
\end{cases}
\eeq
where $z=\sum_{\{e,f\}\in F^1_\bullet} z_{\{e,f\}}$ and $z_{\{e,f\}}=\zeta_{e,f}a_{\{e,f\}}$, for some $\zeta_{e,f}\in \FF_2$.  
In particular, the amalgam is uniquely determined by the pairs $\rho\in F^0$, $\tau\in F^1_\bullet$ with $\rho\le \tau$. To compute $\zeta_{e,f}$, recall that for $j=1,2,\ldots,n$,  
$d^0_{o_j}(e_j)=\sum_{f\in E_{o_j}-\{e_j\}}a_{\{e_j,f\}}$ and so
\beq\label{eqn:zeta}
\zeta_{e,f} =
\begin{cases}
 1 &\text{ if }j\in \delta, {o_j}\in e\cap f\mbox{ and } |\{e,f\}\cap\{e_1,\ldots,e_n\}|=1\\
0 &  \text{ else. } \\
\end{cases}
\eeq
To view the resulting amalgam in the form presented in the theorem, first note that, for any $j$,  $a_{\{e_j,f\}}$ acts as $\gamma_{o_j}\colon s_{o_j}\leftrightarrow  s_{o_j} s_\infty$. 
Now let $i\in I$ and $i\in e,f\in E$.
If $i\in I-\{o_1,\ldots,o_n\}$, then $\varphi_{e,f}^e=\psi_i^e$ is natural inclusion. Moreover if $f,g\in E(\sT)\cap E_{o_j}$, then 
 $e_j\prec f$ so that $\varphi_{e_j,f}^{e_j}=\gamma_{o_j}$ and 
  $\varphi_{e_j,f}^f=\psi_{o_j}^f$, and $\varphi_{f,g}^f=\psi_{o_j}^f$  are inclusion.
However, if $o_j=o_k$, then $\varphi_{e_j,e_k}^{e_j}$ and
 $\varphi_{e_j,e_k}^{e_k}$ are {\em both} the identity, whereas~\eqref{eqn:amalgam parametrization} insists both are twisted by $\gamma_{o_j}$.  Replacing $L_{e_j,e_k}$ by $L_{o_j}^{\gamma_{o_j}}$ induces an isomorphism between the resulting amalgams.
 \qed


\begin{thebibliography}{1}
\bibitem{BloHof2011}
R.~J. Blok and C.~G. Hoffman.
\newblock $1$-cohomology of simplicial amalgams of groups.
\newblock preprint September 2010.

\bibitem{Bru1958}
R.~H. Bruck.
\newblock {\em A survey of binary systems}.
\newblock Ergebnisse der Mathematik und ihrer Grenzgebiete. Neue Folge, Heft
  20. Reihe: Gruppentheorie. Springer Verlag, Berlin, 1958.


\bibitem{Che1974}
O.~Chein.
\newblock Moufang loops of small order. {I}.
\newblock {\em Trans. Amer. Math. Soc.}, 188:31--51, 1974.



\end{thebibliography}
\end{document}